\documentclass[smallextended]{svjour3}       
\smartqed  
\usepackage{graphicx}
\usepackage{mathrsfs,color}
\usepackage{mathptmx, amsmath, amsfonts}

\usepackage{epstopdf}

\newcommand{\cT}{\mathcal{T}}

\newcommand{\Ba}{\partial_t^{1-\alpha}}
\newcommand{\Bac}{{^C}\partial_t^{\alpha}}
\newcommand{\I}{\mathcal{I}}

\begin{document}

\title{Optimal error analysis of a FEM for fractional diffusion problems by energy arguments
\thanks{ The valuable comments of  the referees improved the paper. The support of  the Science Technology Unit at KFUPM through  King Abdulaziz City for Science and Technology (KACST) under National Science, Technology and Innovation Plan (NSTIP) project No. 13-MAT1847-04  is gratefully acknowledged.}}
\titlerunning{FEM for fractional diffusion models}        

\author{Samir Karaa       \and Kassem Mustapha \and         Amiya K. Pani}

\institute{Samir Karaa \at Department of Mathematics and Statistics, Sultan Qaboos University, Al-Khod 123, Muscat, Oman,\email{skaraa@squ.edu.om} \and
Kassem  Mustapha \at Department of Mathematics and Statistics, King Fahd University of Petroleum and Minerals, Dhahran, 31261, Saudi Arabia, \email{kassem@kfupm.edu.sa} \and Amiya K. Pani \at Department of Mathematics, Institute of Technology Bombay, Powai, Mumbai-400076, India, \email{akp@math.iitb.ac.in}}

\date{Received: \today / Accepted: date}
\maketitle

\begin{abstract}
{In this article, the  piecewise-linear finite element method (FEM)  is  applied to approximate the solution of  time-fractional diffusion equations  on bounded convex  domains. Standard energy arguments do not provide satisfactory results for such a problem due to the low regularity of its exact solution.
Using a delicate energy analysis,  {\it a priori} optimal error bounds in $L^2(\Omega)$-, $H^1(\Omega)$-norms, and a quasi-optimal bound  in  $L^{\infty}(\Omega)$-norm are derived for the   semidiscrete FEM  for cases  with smooth and nonsmooth initial data. The main tool  of our analysis is based on a  repeated use of an integral operator and use of a $t^m$ type of weights to take care of the singular behavior of the continuous solution at $t=0.$ The generalized Leibniz formula for fractional derivatives is found to play a key role in our analysis. Numerical experiments are presented  to illustrate some of the theoretical results.}
\end{abstract}
\section{Introduction}
\label{intro}
In this paper, we investigate the error analysis via energy arguments of a semidiscrete Galerkin finite element method (FEM)  for  time-fractional diffusion problems of the form: find $u=u(x,t)$ such that
\begin{equation}\label{a}
u'(x,t) + \Ba {\cal{L}} u(x,t)
=f(x,t) \quad\mbox{ in }\Omega\times (0,T],\end{equation}
with $u(x,0)=u_0(x)$ in $\Omega$, subject to homogeneous Dirichlet boundary conditions, that is, $u(x,t)= 0$ on $\partial\Omega\times (0,T]$. Here ${\cal{L}} u= -{\rm div}({\bf a}(x)\nabla u)$, $\Omega$ is a bounded, convex polygonal domain in  $\mathbb{R}^2$  with boundary $\partial \Omega$, $f$,  ${\bf a}$ and  $u_0$ are given functions defined on their respective domains. In \eqref{a}, $u'$ is the time partial derivative of $u$  and $\Ba:={^R}{\rm D}^{1-\alpha}$ is the Riemann--Liouville time-fractional derivative  defined by: for $0<\alpha<1$,
\begin{equation} \label{Ba}
\Ba \varphi(t):=\frac{\partial }{\partial t}\I^{\alpha}\varphi(t):=\frac{\partial }{\partial t}\int_0^t\omega_{\alpha}(t-s)\varphi(s)\,ds\quad\text{with} \quad \omega_{\alpha}(t):=\frac{t^{\alpha-1}}{\Gamma(\alpha)},
\end{equation}
($\I^{\alpha}$ is the Riemann--Liouville time-fractional integral).  We assume that the source term $f$ and the diffusivity coefficient function  ${\bf a}$ are  sufficiently regular and
\begin{equation}\label{eq: A positive}
0<a_{\min} \le  {\bf a}(x) \le  a_{\max}<\infty\quad { \rm on
} ~~\overline \Omega.\end{equation}

  Several numerical techniques  for the problem \eqref{a} (with constant diffusivity coefficient) in one and several space variables have been proposed with various types of spatial discretizations  including finite difference, finite volume   or spectral element methods, see \cite{CockburnMustapha2015,KMP2017,JLPZ2015}.  For the time discretization,  different time-stepping schemes (implicit and explicit) have been investigated including finite difference, convolution quadrature,  and discontinuous Galerkin methods, see \cite{ChenLiuAnhTurner2012,CuestaLubichPalencia2006,Cui2012,Mustapha2011,RZ2013,ZS2011}. The error analyses  in most studies in the existing  literature typically assume that the solution $u$ of \eqref{a}   is sufficiently regular  including at $t = 0$, which is not practically the case, see \cite{Mclean2010}. Indeed, assuming high  regularity on $u$ imposes additional compatibility conditions on the given data, which are not reasonable in many  cases.

Though the numerical approximation of the solution $u$ of  \eqref{a} was considered by many authors over the last decade, the optimality of the estimates with respect to the solution smoothness expressed through the problem data, $f$ and $u_0$, was considered in a few papers for the case of constant diffusivity and quasi-uniform meshes. Obtaining sharp error bounds under reasonable regularity assumptions on $u$ has proved challenging.  The first optimal $L^2(\Omega)$-error estimate for the  Galerkin finite  element solution of \eqref{a} with respect to the regularity of initial data was established \cite{MT2010a}. More precisely, for $t\in(0,T]$, convergence rates of order $h^2 t^{\alpha(\delta-2)/2} $ ($h$ denoting the maximum diameter of the spatial mesh elements) were proved assuming that the initial data  $u_0 \in \dot H^\delta(\Omega)$ for $\delta=0,\,2$ (see, Section \ref{sec:WRT} for the definition of these spaces). The proof was based on some refined estimates of the Laplace transform in time for the error.  In \cite{MT2010b}, by using a similar approach, the same authors derived $O(h^2\ell_h^2t^{-\alpha(2-\delta)/2})$ convergence rates  in the stronger $L^\infty(\Omega)$-norm, where $\ell_h=\max\{1,|\log h|\}.$ For $\delta=0$,  $u_0$ was assumed to be in  $L^\infty(\Omega)$, while  for $\delta=2$, $u_0$ was assumed to be in  $C^2(\Omega)$ and vanishes on $\partial \Omega$. Recently, in \cite{McLeanMustapha2015},  the error analysis of a first order semidiscrete time-stepping scheme for problem \eqref{a} with $f\equiv 0$ allowing nonsmooth  $u_0,$ using discrete Laplace transform technique. Since standard energy arguments are used heavily in the error analysis of Galerkin FEMs for classical diffusion equations,  it is more pertinent to extend the analysis to these time-fractional order diffusion problems with a variable diffusivity. Since $t^m$ and $\Ba{\cal{L}} $ do not commute, extending these arguments  to problem \eqref{a} is not a straightforward task, especially in the case of nonsmooth $u_0$.

The main motivation of this work is to propose delicate energy arguments approach to derive optimal error estimates of the semidiscrete  Galerkin FEM  for the problem   \eqref{a} for both smooth and nonsmooth initial data $u_0$. Earlier,  for smooth $u_0$, a quasi-optimal error estimate of order $O(h^2 \ell_h)$ in $L^{\infty} (L^2)$-norm  was derived  in \cite{MustaphaMcLean2011} using direct energy arguments. The proposed technique in this work has several advantages over the approaches  used  in  \cite{MT2010a,MT2010b} to show optimal error bounds for smooth and nonsmooth $u_0$. Some of these are: (1) allowing variable coefficients, (2) the  source term $f$ can depend on the unknown solution $u$, that is, the fractional diffusion problem \eqref{a} is semilinear  (see Remark \ref{rem: semilinear}), (3) the quasi-uniform mesh assumption is not required to show the convergence results in $H^m$-norm for $m=0,1$, and (4) the proposed energy argument approach can be applied to other fractional diffusion problems.   For instance, the achieved error bounds in Theorems \ref{thm: smooth and nonsmooth}, \ref{thm: H1 bound} and \ref{thm: smooth and nonsmooth-5} can be extended  to the  time-fractional diffusion equation,
\begin{equation}\label{D}
\Bac u(x,t)  + {\cal{L}} u(x,t)= f(x,t)\quad {\rm for}~~0<\alpha<1,
\end{equation}
where $\Bac v(t):=\I^{1-\alpha}v'(t)$ is the  Caputo derivative. The error estimate in Theorem \ref{thm: smooth and nonsmooth} provides an improvement of the result obtained   in \cite[Theorem 3.7]{JLZ2013}. For constant diffusivity, using a semigroup approach and assuming  that the mesh is quasi-uniform, the derived error bound therein  involves a logarithmic factor.

 Outline of the paper. In Section \ref{sec:WRT}, we recall some smoothness properties of the solution $u$, we also state and derive some technical results.
In Section \ref{sec:Semi-discrete FE}, we introduce our semidiscrete finite element  scheme and recall some  error results. We claim that a direct application of energy arguments to problem \eqref{a} does not lead to optimal convergence rates even when the initial data $u_0 \in \dot H^2(\Omega)$.    In Section \ref{sec: LinftyL2}, for $t\in (0,T]$ and for $u_0 \in \dot H^\delta(\Omega)$,  an optimal error estimate in $L^2(\Omega)$-norm of order $h^2t^{-\alpha(2-\delta)/2}$ is established  when  $0\le \delta\le 2$, see Theorem \ref{thm: smooth and nonsmooth}. In Section \ref{sec: LinftyH1}, a superconvergence gradient  error bound is obtained, see Theorem \ref{sup-conv-1}. As a consequence, an optimal  $O(h\,t^{-\alpha(2-\delta)/2})$  estimate in the $H^1(\Omega)$-norm is derived for $0\le \delta\le 1,$ see Theorem \ref{thm: H1 bound}.   However, for $1<\delta\le 2,$ we showed an $O(h\,t^{-\alpha(1-\delta)/2}\max\{1,(h/t)^{-\alpha(1-\delta)/2}\})$ error estimate, reduces to $h\,t^{-\alpha(2-\delta)/2}$  for $t=O(h)$. Furthermore,
 assuming that $u_0 \in \dot H^\delta(\Omega)\cap L^\infty(\Omega)$ and the mesh is quasi-uniform,  a quasi-optimal error estimate of order $h^2\ell_h^{5/2}t^{-\alpha(3-\delta)/2}$ in  the stronger $L^\infty(\Omega)$-norm is proved in Theorem \ref{thm: smooth and nonsmooth-5}. Particularly relevant to this {\it a priori} error analysis is the appropriate use of several  properties of the time-fractional integral and derivative operators. Numerical tests are presented in Section \ref{sec: Numerical}  to confirm some of our theoretical findings.  Throughout the paper, $C$ is a generic positive constant that may depend on $\alpha$ and $T$, but is independent of  $h$.

\section {Regularity and technical results}\label{sec:WRT} Smoothness properties of the solution $u$ of the fractional diffusion problem \eqref{a} play a key role in the error analysis  of the Galerkin FEM, particularly, since $u$  has singularity near $t=0$, even for smooth given data. Below, we state the required regularity results for problem \eqref{a} in terms of the initial data $u_0$ and the source term $f$. For  $0\le r,\,\mu \le 2$,
\begin{equation}\label{eq: regularity property}
t^q\|u^{(q)}(t)\|_{r+\mu} \le  C(1+T^{\alpha\mu/2}) \,t^{-\alpha\mu/2}d_r(u_0,f),\quad  {\rm for}~~q\in \{0,1\},
\end{equation}
with $r+\mu\le 2$, where $d_r(u_0,f)=\|u_0\|_{r} +\sum_{m=0}^{q+1}\int_0^T s^m\|f^{(m)}(s)\|_{r}\,ds.$  Here, $\|\cdot\|_r$ denotes the norm on the Hilbert space $\dot H^r(\Omega)$ defined by $\|v\|_r^2  =\sum_{j=1}^\infty \lambda_j^r (v,\phi_j)^2,$ where $\{\lambda_j\}_{j=1}^\infty$ are the eigenvalues of the
elliptic operator ${\cal{L}}$ (subject to homogeneous Dirichlet boundary conditions) and $\{\phi_j\}_{j=1}^\infty$ are the associated orthonormal eigenfunctions. Noting that, $\dot H^r(\Omega)=H^r(\Omega)$ for $0\le r<1/2,$ $\dot H^r(\Omega)=\overline {C^\infty_0(\Omega)}$   in $H^r(\mathbb{R}^2)$ for $r=1/2,$ and for convex polygonal domains, $\dot H^r(\Omega)=\{w \in H^r(\Omega):  w=0~{\rm on}~\partial \Omega\}$ for $1/2< r \le 2,$ where $H^r(\Omega)$ is the standard Sobolev space with $H^0(\Omega)=L^2(\Omega)$.

For constant diffusivity, over the convex domain $\Omega$, the regularity property \eqref{eq: regularity property} follows by combining the results of Theorems 4.1, 4.2   and 5.6 in \cite{Mclean2010}. In the proof, it was used that the operator ${\cal L}$ (subject to homogeneous Dirichlet) is positive definite and possess a complete eigensystem. These properties remain valid if the diffusivity coefficient function ${\bf a}$ is sufficiently regular and  satisfies the positivity assumption \eqref{eq: A positive}.

Next, we state the positivity  properties of the fractional  operators  $\I^{\alpha}$ and $\Ba$, and derive some technical results that will be used in the subsequent sections. By \cite[Lemma 3.1 (ii)]{MustaphaSchoetzau2014},  and since the  bilinear form $A(\cdot, \cdot)$ associated with the operator ${\cal{L}}$ (that is, $A(v,w)=({\bf a}\nabla v, \nabla w)$) is symmetric positive definite on the Sobolev space $H_0^1(\Omega)$, it follows that for piecewise time continuous functions  $\varphi:[0,T] \to H_0^1(\Omega),$
\begin{equation}\label{eq: positive of Ia}
\int_0^TA(\I^\alpha\varphi,\varphi)\,dt\ge \cos(\alpha \pi/2)\int_0^T\|\sqrt{ \bf a}\,\nabla \I^{\alpha/2}\varphi\|^2\,dt \ge  0~~{\rm  for}~~0<\alpha<1,
\end{equation}
where  $\|\varphi\|:=\sqrt{(\varphi,\varphi)}$ denotes the $L^2(\Omega)$-norm. Furthermore, by \cite[Lemma A.1]{McLean2012},  the following holds: for
$W^{1,1}(0,T;H_0^1(\Omega))$,
\begin{equation}\label{eq: positive of Ba}
\int_0^TA(\Ba\varphi(t),\varphi(t))\,dt \ge \frac{1}{2}\sin(\alpha\pi/2) T^{\alpha-1} \int_0^T\| \sqrt{{\bf a}} \nabla \varphi(t)\|^2\,dt\,.
\end{equation}

The next lemma will be used frequently in our convergence analysis. In the proof, we use the following integral inequality:
if  for any $\tau\in(0,t)$,  $ |\phi(\tau)|^2\le |\phi(0)|^2+2\,\int_{0}^\tau|\phi(s)|\,|\psi(s)|\,ds, $ then $|\phi(t)|\le |\phi(0)|+\int_0^t\,|\psi(s)|\,ds$, see \cite[Lemma 4]{CockburnMustapha2015}.
\begin{lemma}\label{lem: reg use} Let $\kappa \in \{0,1\}$ and let ${\mathcal B}^\alpha=\Ba$ or ${\mathcal B}^\alpha=\I^\alpha.$   Assume that
\begin{equation} \label{eq: reg 1}
\kappa(v(t),\chi)+(1-\kappa)(v'(t),\chi)+A({\mathcal B}^\alpha v(t),\chi)= (w(t),\chi),\quad \forall~ \chi \in V_h,
\end{equation}
for  $t \in (0,T]$. Then
\[ \kappa\int_0^t \|v\|^2\,ds+(1-\kappa)\|v(t)\|^2\le
(1-\kappa)\Big(\|v(0)\|+\int_0^t \|w\|\,ds\Big)^2+\kappa\int_0^t \|w\|^2\,ds.\]
\end{lemma}
\begin{proof} Choose $\chi=v$ in \eqref{eq: reg 1}, and then, integrate over the interval $(0,t)$ to obtain
\[ 2\kappa\int_0^t \|v\|^2\,ds+(1-\kappa)[\|v(t)\|^2-\|v(0)\|^2]+2\int_0^t A({\mathcal B}^\alpha v,v)ds=
2\int_0^t (w,v)\,ds.\]
By the  positivity properties   in \eqref{eq: positive of Ia} and in  \eqref{eq: positive of Ba},  $\int_0^t A({\mathcal B}^\alpha v,v)ds\ge 0$, and thus,
\[2\kappa\int_0^t \|v\|^2\,ds+(1-\kappa)\|v(t)\|^2\le (1-\kappa)\|v(0)\|^2+2\int_0^t \|w\| \,\|v\|\,ds.
\]
Therefore, for $\kappa=0$, an application of the integral inequality (stated above) yields the desired inequality. However, for $\kappa=1$, we use the inequality $2\|w\| \,\|v\|\le \|w\|^2+ \|v\|^2$ and the desired result follows.$\quad \Box$
\end{proof}

\section {Semi-discrete FEM} \label{sec:Semi-discrete FE} This section focuses on a semidiscrete Galerkin FEM for  problem \eqref{a}. To define the scheme, let  $\mathcal{T}_h$ be a family of regular triangulations (made of simplexes $K$) of the  domain $\overline{\Omega}$ and let $h=\max_{K\in \mathcal{T}_h}(\mbox{diam}K),$ where $h_{K}$ denotes the diameter  of the element  $K.$  Let $V_h \subset H^1_0(\Omega)$ denote the usual space of continuous, piecewise-linear functions on  $\mathcal{T}_h$ that vanish on $\partial \Omega$.

The weak formulation for problem  \eqref{a} is to find   $u:( 0,T]\longrightarrow H^1_0(\Omega)$ such that
\begin{equation} \label{weak}
(u',v )+ A(\Ba u,v )=  (f,v )\quad \forall v\in H^1_0(\Omega)
\end{equation}
with given  $ u(0)=u_0.$ Thus, the standard semidiscrete finite element  formulation for \eqref{a} is to seek  $u_h:(0,T]\longrightarrow V_h$ such that
\begin{equation} \label{semi}
(u_h',v_h)+ A(\Ba u_h,v_h)=  (f,v_h)\quad \forall v_h\in V_h
\end{equation}
with given  $u_h(0)\in V_h$ to be defined later.

To derive  {\it a priori} error estimates for the numerical scheme (\ref{semi}),  we split the error $e:=(u- R_h u)-( u_h-R_hu)=:\rho-\theta,$ where the Ritz projection $R_h : H_0^1(\Omega) \rightarrow V_h $ is defined by the following relation: $A(R_h v-v, \chi)= 0$ for all $\chi\in V_h.$ For $t\in (0,T]$, the projection errors $\rho(t)$ and $\rho'(t)$  satisfy the following estimates: for $j=0,1,$
\begin{equation}\label{rho-estimate}
\|\rho(t)\|_j\leq C h^{m-j} \|u(t)\|_m~~{\rm and}~~ \|\rho'(t)\|_j\leq C h^{m-j} \|u'(t)\|_m,~~{\rm for}~~m=1,2.
\end{equation}
Hence,  by using the regularity property in \eqref{eq: regularity property}, we observe
\begin{equation}\label{eq: rho}
\|\rho(t)\|+t\|\rho'(t)\| \le Ch^m t^{-\max\{0,\alpha(m-\delta)/2\}}d_\delta(u_0,f),\quad {\rm for}~~0\le \delta\le 2.
\end{equation}

Next, we show that a direct application of energy arguments to problem (\ref{a}) does not  yield satisfactory results due to the low regularity of the continuous solution. From (\ref{weak}) and (\ref{semi}), the error decomposition $e=\rho-\theta$, and the property of the elliptic projection, we obtain the  equation in $\theta$ as
\begin{equation} \label{sup-2}
(\theta',\chi)+A(\Ba\theta,\chi)=(\rho',\chi)\quad \forall~\chi \in V_h.
\end{equation}
Then, the following result  holds \cite[Theorem 4]{MustaphaMcLean2011}: for $t\in (0,T],$ we have
\begin{equation}\label{e1}
\|u(t)-u_h(t)\|\leq   \|\theta(0^+)\|+\int_0^t \|\rho'(s)\|\,ds+\|\rho(t)\|.
\end{equation}
Practically,  the  solution  $u$ has singularity near $t=0$. For instance,  if $f\equiv 0$ and $u_0 \in \dot H^2(\Omega)$, then  $\|u'(t)\|_m\le Ct^{\alpha(2-m)/2-1}\|u_0\|_2$ for $m=1,2$, see \cite{Mclean2010}. Hence, by (\ref{rho-estimate}),
\begin{align*}
\int_0^t \|\rho'(s)&\|\,ds \le Ch\int_0^\epsilon \|u'(s)\|_1\,ds+Ch^2\int_\epsilon^t \|u'(s)\|_2\,ds\\
&\le C\Big(h\int_0^\epsilon s^{\alpha/2-1}\,ds+h^2\int_\epsilon^t s^{-1}\,ds\Big)\|u_0\|_2 \le Ch^2\ell_h\|u_0\|_2,~{\rm  for}~\epsilon=h^{2/\alpha}.
 \end{align*}
This leads to a quasi-optimal $O(h^2\ell_h)$ convergence. To achieve an optimal $O(h^2)$  convergence,  a stronger regularity assumption on $u$ is required and that, in turn  imposes severe restrictions on the initial data $u_0$. Thus, the error  bound in \eqref{e1} is not sharp even for the case of smooth $u_0$, that is, $H^2$-regularity on $u_0$ is not sufficient to get an  optimal $O(h^2)$ convergence rate. Furthermore, it is clear that this upper bound is not suitable for the case of nonsmooth $u_0$. Therefore, we propose in the next section  an approach via delicate energy arguments that provides optimal error bounds for both cases:  smooth and nonsmooth $u_0$.

\begin{remark}\label{rem: semilinear}
Our forthcoming convergence analysis can be easily extended if the source term $f=f(x,t,u(x,t))$ in the problem \eqref{a}, assuming that $f$ is sufficiently regular in the three variables and satisfies that $|f(x,t,z_1)-f(x,t,z_2)|\le C|z_1-z_2|$ for $z_1,z_2 \in \mathbb{R}$ (that is, $f$ is  Lipschitz continuous in the third variable). Studying the regularity properties of the continuous solution $u$ remains an open problem in this case.

The spatial finite element  scheme for \eqref{a}  is: find   $u_h:(0,T]\longrightarrow V_h$ such that
\[(u_h',v_h)+ A(\Ba u_h,v_h)=  (f(u_h),v_h)\quad \forall ~v_h\in V_h\]
with given  $u_h(0)\in V_h$. Hence, instead of \eqref{sup-2}, we have
\[(\theta',\chi)+A(\Ba\theta,\chi)= (\rho'+[f(u_h)-f(u)],\chi)\quad \forall~\chi \in V_h.\]
We follow the proofs in Sections \ref{sec: LinftyL2} and \ref{sec: LinftyH1} step-by-step where the term $\rho(t)$ will be replaced with $\tilde \rho(t):=\rho(t)+\int_0^t[f(u_h)-f(u)]\,ds$ and the Lipschitz continuity property of $f$ will be used appropriately.$\quad \Box$
\end{remark}
\section {$L^2(\Omega)$-error estimates} \label{sec: LinftyL2}
For convenience, we introduce the  notations:
\[\Theta_i(t):=t^i\theta(t)\quad{\rm and}\quad  \dot \Theta_i(t):=t^i\theta'(t)~~{\rm for}~~i=1,\,2.\]
In the next lemma, based on the generalized Leibniz formula for fractional derivatives, we state and show some identities
for our subsequent use.
\begin{lemma}\label{Ia} For $0<\alpha<1$, the followings hold: \\
\mbox{(a)} $t\Ba\theta= \Ba\Theta_1-(1-\alpha)\I^\alpha \theta$,\\
\mbox{(b)} $t\I^\alpha\theta = \I^\alpha\Theta_1+\alpha\I^{1+\alpha}\theta$.
\end{lemma}
\begin{proof} The first identity follows from the fractional Leibniz formula. To show the second identity,  noting first  that $\Ba\Theta_1=\I^\alpha \Theta'_1=\I^\alpha\theta+\I^\alpha \dot \Theta_1.$ Hence, by (a),
\begin{equation}\label{k3a}
t\Ba \theta(t)= \I^\alpha \dot \Theta_1(t)+\alpha\I^\alpha\theta(t).
\end{equation}
Now, we replace $\theta$ by $\I\theta$ in (\ref{k3a}) to obtain the second identity in the lemma.$\quad \Box$
\end{proof}

Next, we derive an upper bound of $\Theta_1$. To do so, we let $u_h(0)=P_h u_0,$ where $P_h :L^2(\Omega)\rightarrow V_h$ denotes the $L^2$-projection defined by $(P_h v-v, \chi)= 0$ for all $\chi\in V_h.$
\begin{lemma}\label{lem: estimate of Theta in l2 norm}
Let $u_h(0)=P_h u_0.$ Then, we have
\[\int_0^t\|\Theta_1\|^2\,ds \le 3\int_0^t \Big( s^2\|\rho\|^2+ 2\big(\I\|\rho\|\big)^2\Big)\,ds,\quad {\rm for}~~t\in(0,T].\]
\end{lemma}
\begin{proof} We integrate \eqref{sup-2} over the time interval $(0,t)$ and obtain
\begin{equation} \label{sup-3}
(\theta,\chi)+A(\I^\alpha\theta-\I^\alpha\theta(0^+),\chi)=(\rho+e(0),\chi)\quad \forall~\chi \in V_h.
\end{equation}
However, $A(\I^\alpha\theta(0^+),\chi)=-A(\I^\alpha e(0),\chi)$ for  $\chi \in V_h$, and $\I^\alpha e(0)=0$ because $u$ and $u_h$ are both continuous on the time interval $[0,T]$. Furthermore, $(e(0),\chi)= 0$ due to the equality  $u_h(0)=P_h u_0.$ Therefore,
\begin{equation} \label{sup-3 ph}
(\theta,\chi)+A(\I^\alpha\theta,\chi)=(\rho,\chi)\quad \forall~\chi \in V_h.
\end{equation}
Multiply by $t$ and use $t\I^\alpha\theta= \I^\alpha\Theta_1+\alpha\I^{1+\alpha}\theta$ by Lemma \ref{Ia}\,(b) to find that
\[(\Theta_1,\chi)+A(\I^\alpha\Theta_1,\chi)=
t(\rho,\chi)-\alpha A(\I^{1+\alpha}\theta,\chi)\quad \forall~\chi \in V_h.\]
However, from  \eqref{sup-3 ph}, we get
\begin{equation}\label{eq: I alpha+1}
A(\I^{\alpha+1}\theta,\chi)=(\I(\rho-\theta),\chi)\quad \forall~\chi \in V_h,
\end{equation}
and thus,
\begin{equation} \label{sup-3-2}
(\Theta_1,\chi)+A(\I^\alpha\Theta_1,\chi)=t(\rho,\chi)-\alpha (\I(\rho-\theta),\chi)\quad \forall~\chi \in V_h.
\end{equation}
Consequently, an application of Lemma \ref{lem: reg use}  (with $\kappa=1$) yields
\begin{equation}\label{eq: second last step}
\int_0^t\|\Theta_1\|^2\,ds \le \int_0^t \|s\rho-\alpha\I(\rho-\theta)\|^2\,ds \le 3\int_0^t ( s^2\|\rho\|^2+ \|\I\rho\|^2+\|\I\theta\|^2)\,ds.
\end{equation}
To complete our proof, we rewrite  \eqref{eq: I alpha+1} as
\[(\I\theta,\chi)+A(\I^\alpha(\I\theta),\chi)=(\I\rho,\chi)\quad \forall~\chi \in V_h,\]
Again, an application of Lemma \ref{lem: reg use}  (with $\kappa=1$) shows
\begin{align}\label{eq: theta less rho}
\int_0^t \|\I\theta\|^2\,ds &\le  \int_0^t \|\I\rho\|^2\,ds.
\end{align}
Substitute (\ref{eq: theta less rho}) in \eqref{eq: second last step} yields the desired bound.$\quad \Box$
\end{proof}

An upper bound of the term $\theta$ will be derived  in the next lemma. Again, for convenience, we introduce the following notation
\begin{equation}\label{def: B1}
{\mathcal B}_1(t):=\int_0^t \Big(  s^4\|\rho'(s)\|^2+s^2\|\rho(s)\|^2+ 2\big(\I\|\rho(s)\|\big)^2 \Big)\,ds\,.\end{equation}
For later use, by using the projection error estimates in \eqref{eq: rho} (with $m=2$) for upper bounds of $\rho$ and $\rho'$, and then integrating, we find that for $t\in (0,T]$,
\begin{equation}\label{eq: bound of B1}
{\mathcal B}_1(t)\leq C\,h^4 t^{3-\alpha(2-\delta)}d_\delta^2(u_0,f),\quad {\rm for}~~0\le \delta\le 2\,.
\end{equation}
 \begin{lemma} \label{H1}
 Let  $u_h(0)=P_h u_0$. Then,  the following estimate holds
\[\|\theta(t)\|^2 \leq C\, t^{-3} {\mathcal B}_1(t),\quad {\rm for}~~t\in (0,T]\,.\]
\end{lemma}
\begin{proof} We multiply  (\ref{sup-2}) by $t^2$ so that
\begin{equation}\label{k-3}
(\dot\Theta_2,\chi)+A(t^2\Ba\theta,\chi)=(t^2\rho', \chi),
\end{equation}
where $\dot\Theta_2= t^2 \theta'.$ From the fractional  Leibniz formula, we have
\begin{align*}
t^2\Ba\theta&=\Ba\Theta_2- 2(1-\alpha)t\I^{\alpha} \theta+\alpha(1-\alpha)\I^{1+\alpha}\theta.
\end{align*}
Hence, we  rearrange (\ref{k-3}) as
\begin{equation}\label{k2-1}
(\dot\Theta_2,\chi)+A(\Ba\Theta_2,\chi)=(t^2\rho', \chi)+ (1-\alpha)\Big(2tA(\I^{\alpha} \theta,\chi)-\alpha A(\I^{1+\alpha}\theta,\chi)\Big),
\end{equation}
and then, by equations  \eqref{sup-3 ph} and \eqref{eq: I alpha+1},
\begin{equation}\label{k3}
(\Theta_2',\chi)+A(\Ba\Theta_2,\chi)=(t^2\rho'+  2\alpha \Theta_1+  (1-\alpha)(2t\rho-\alpha\I (\rho-\theta)),\chi).
\end{equation}
Hence, by  Lemma  \ref{lem: reg use}  (with $\kappa=0$),  we obtain
\begin{align*}
\|\Theta_2(t)\|  &\leq \int_0^{ t} \Big(s^2\|\rho'(s)\|+ 2s\|\rho(s)\|+2\|\Theta_1(s)\|+\|\I(\rho-\theta)\|\Big)\,ds\,,
\end{align*}
and thus, an application of the Cauchy-Schwarz inequality yields
\begin{align*}
\|\Theta_2(t)\|^2 &\leq Ct\int_0^t \Big( s^4\|\rho'(s)\|^2+s^2\|\rho(s)\|^2+\|\Theta_1(s)\|^2+\|\I\rho\|^2+\|\I\theta\|^2\Big)\,ds\,.
\end{align*}
Therefore, by using the identity  $\theta(t)=t^{-2}\Theta_2(t)$, the inequality in \eqref{eq: theta less rho} and Lemma \ref{lem: estimate of Theta in l2 norm} will complete the rest of the proof. $\quad \Box$
\end{proof}

In the next theorem, we derive optimal convergence results   of the finite element  \eqref{semi} in the $L^2(\Omega)$-norm  for both smooth and nonsmooth initial data $u_0$. For $u_0 \in \dot H^{\delta}(\Omega)$ with $0\le \delta\le 2,$ we show that the error is bounded by $C h^2 t^{-\alpha(2-\delta)/2}$ for each  $t\in (0,T]$. Recall that,  $\dot H^{\delta}(\Omega)=\{v\in H^{\delta}(\Omega):~v=0~{\rm on}~\partial \Omega\}$ for $1/2<\delta\le 2$, while $\dot H^{\delta}(\Omega)=H^{\delta}(\Omega)$ for $0\le \delta <1/2$.
 \begin{theorem} \label{thm: smooth and nonsmooth}
 Let $u$ and $u_h$  be the solutions of $(\ref{a})$ and $(\ref{semi})$,  respectively, with $u_h(0)=P_h u_0$. Then,
\[ \|(u-u_h)(t)\| \leq
 C h^2 t^{-\alpha(2-\delta)/2}d_\delta(u_0,f)\quad{\rm for}~~t \in (0,T]~~{\rm with}~~0\le \delta \le 2\,.\]
\end{theorem}
\begin{proof} The desired result follows from the decomposition $u-u_h=\rho-\theta$, the estimate of $\theta$ in Lemma \ref{H1}, the bound  in \eqref{eq: bound of B1}, and the estimate of $\rho$  in \eqref{eq: rho}.  $\quad \Box$
\end{proof}
\begin{remark}\label{rem: 0}
In the proof of the above theorem, we used \eqref{eq: bound of B1} which follows from the projection estimate in \eqref{eq: rho} for $m=2.$ For $m=1$, we follow similar steps where  $2-\delta$ will be replaced with $1-\delta$, to obtain ${\mathcal B}_1(t)\leq C\,h^2 t^{3-\alpha(1-\delta)}d_\delta^2(u_0,f)$ for $0\le \delta\le 1\,.$ Now, for $1<\delta\le 2$, we notice first that for $0\le q\le t$ with $t_h=\max\{t,h\}$,
\[\|\rho(q)\|\le \|\rho(t_h)-\rho(q)\|+\|\rho(t_h)\|\le \I(\|\rho'(t_h)\|)+\|\rho(t_h)\|\,.\]
  Substituting this in the definition of ${\mathcal B}_1$ defined in \eqref{def: B1}, we observe
\[{\mathcal B}_1(t)\le Ct^2\int_0^t s^2\|\rho'(s)\|^2\,ds+Ct^3\Big(\I(\|\rho'(t_h)\|) +\|\rho(t_h)\|\Big)^2\,.\]
To estimate the first two terms, we use \eqref{rho-estimate} (with $m=1$) and the following regularity property (which follows from  \cite[Theorems 4.2 and 5.6]{Mclean2010})
\begin{equation}\label{eq: regularity property-2}
t\|u'(t)\|_1\le  Ct^{\alpha (r-1)/2}\tilde d_r(u_0,f),\quad t\in(0,T],\quad{\rm for}~~0\le r\le 2,
\end{equation}
where $\tilde d_r(u_0,f)=\|u_0\|_r +\sum_{j=0}^2\int_0^T s^{j+\alpha(1-r)/2}\|f^{(j)}(s)\|_r\,ds,$ we arrive to
\[{\mathcal B}_1(t)\le Ct^3 h^2 t_h^{\alpha(\delta-1)}\tilde d_\delta(u_0,f) +Ct^3\|\rho(t_h)\|^2,\quad{\rm for}~~1<\delta\le 2\,.\]
However, by \eqref{eq: rho} and the inequality $t_h^{-\alpha/2}\le h^{-\alpha/2}$, we find that
\[\|\rho(t_h)\| \le Ch^2 t_h^{\alpha(\delta-2)/2}d_\delta(u_0,f)\le  C h^{2-\alpha/2} t^{\alpha(\delta-1)/2}(t_h/t)^{\alpha(\delta-1)/2}d_\delta(u_0,f).\]
Therefore,
\[{\mathcal B}_1(t)\le C h^2 t^{3+\alpha(\delta-1)}(t_h/t)^{\alpha(\delta-1)}\tilde d^2_\delta(u_0,f),\quad{\rm for}~~1<\delta\le 2\,.\]
Consequently, by using the above bound of ${\mathcal B}_1$ in Theorem \ref{thm: smooth and nonsmooth}, we get  the error estimate below that will be used in the forthcoming section to show the convergence of the gradient finite element  solution:
 \begin{equation} \label{thm: smooth and nonsmooth m=1}
 \|e(t)\| \leq  C h\, t^{-\alpha(1-\delta)/2}D_{\delta,\alpha}(u_0,f,h/t),\quad {\rm for}~~~t \in (0,T],
\end{equation}
where $D_{\delta,\alpha}(u_0,f,h/t)= \tilde d_\delta(u_0,f)$ for $0\le \delta \le 1$, while for $1< \delta \le 2,$
 $D_{\delta,\alpha}(u_0,f,h/t)= \tilde d_\delta(u_0,f)(t_h/t)^{\alpha(\delta-1)/2}$. $\quad \Box$
\end{remark}
\begin{remark}\label{rem: 1}
Under the quasi-uniformity condition on $V_h$, for $t \in (0,T]$, from the decomposition  $u-u_h=\rho-\theta$, the inverse inequality, the estimate of $\theta$ in Lemma \ref{H1}, and the estimate $\|\rho(t)\|_1 \le  Ch\|u(t)\|_2\le Ct^{-\alpha(2-\delta)/2}d_\delta(u_0,f)$ (follows from the Ritz projection bound in \eqref{rho-estimate} with $j=1$ and $m=2$  and the regularity property \eqref{eq: regularity property}), we obtain the following optimal error estimate:
\[ \|\nabla(u-u_h)(t)\| \leq C h\, t^{-\alpha(2-\delta)/2}d_\delta(u_0,f)\quad{\rm for}~~t \in (0,T]~~{\rm with}~~0\le \delta \le 2\,.\]
 This error bound remains valid for $0\le \delta\le 1$ in the absence of the quasi-uniformity mesh assumption,  see Theorem \ref{thm: H1 bound}.$\quad \Box$
\end{remark}

   \begin{remark}\label{rem: 2} For smooth initial data $u_0\in {\dot H}^2(\Omega)$, one may choose $u_h(0)=R_hu_0$. An optimal convergence rate can be shown by following the proof of Theorem \ref{thm: smooth and nonsmooth} line-by-line, where the term $\rho$ in Lemma \ref{lem: estimate of Theta in l2 norm} should be replaced with $\tilde \rho:=\rho+e(0)$. $\quad \Box$ \end{remark}

\section {$H^1(\Omega)$- and $L^\infty(\Omega)$-error estimates}\label{sec: LinftyH1}
In this section, we show optimal convergence error results  in the  $H^1(\Omega)$-norm, and quasi-optimal error bounds  in the  $L^\infty(\Omega)$-norm, for both smooth and nonsmooth initial data $u_0.$ We start our analysis by deriving an upper bound of   $\nabla \Theta_1$.
\begin{lemma}\label{theta}
For $0\le \delta \le 2$ and for  $t\in(0,T]$, we have
\[\int_0^{t} \|\nabla\Theta_1\|^2ds \leq C h^4 t^{3 -\alpha(3-\delta)}d^2_\delta(u_0,f)\,.\]
 \end{lemma}
\begin{proof} Multiplying  \eqref{sup-2} by $t$ and  then using the first identity in Lemma \ref{Ia},
\begin{equation}\label{k2}
(\dot \Theta_1,\chi)+A(\Ba\Theta_1,\chi)=(t\rho', \chi)+ (1-\alpha)A(\I^\alpha\theta,\chi).
\end{equation}
Then, a use of  (\ref{sup-3 ph})  yields after simplifying
\begin{equation}\label{k3-1}
(\Theta_1',\chi)+A(\Ba\Theta_1,\chi)= ((t\rho)',\chi)-\alpha(e, \chi)\,.
\end{equation}
Now, set $\chi=\Theta_1$ in \eqref{k3-1}, integrate the resulting equation over $(0,t)$, and use the positivity property  of $\Ba$ in  \eqref{eq: positive of Ba},  to find that
\begin{equation}\label{eq: bound theta 12}\|\Theta_1(t)\|^2+\frac{1}{2}\sin(\alpha\pi/2) t^{\alpha-1}
\int_0^t\|\sqrt{{\bf a}} \nabla \Theta_1\|^2\,ds \le \int_0^t (\|(s\rho)'\|+\|e\|)\|\Theta_1\|\,ds.\end{equation}
 This implies
\begin{align*}\|\Theta_1(t)\|^2 &\le \int_0^t (\|(s\rho')\|+\|e\|)\|\Theta_1\|\,ds.\end{align*}
By the integral inequality (stated before Lemma \ref{lem: reg use}), we observe
\begin{align*}\|\Theta_1(t)\| &\le \frac{1}{2}\int_0^t \big(\|(s\rho)'\|+\|e\|\big)\,ds.\end{align*}
Substitute this bound in the RHS of \eqref{eq: bound theta 12} yields
\begin{equation}\label{eq: bound theta}
\sin(\alpha\pi/2) t^{\alpha-1} \int_0^t\| \sqrt{{\bf a}}\nabla \Theta_1\|^2\,ds \le  \Big(\int_0^t \big(\|(s\rho)'\|+\|e\|\big)ds\Big)^2.
\end{equation}
Therefore, the desired   estimate follows from this bound, the error projection in \eqref{eq: rho} (with $m=2$), the convergence results in Theorem \ref{thm: smooth and nonsmooth}, and   \eqref{eq: A positive}.$\quad \Box$
\end{proof}

In the next theorem, we derive  an error bound for $\nabla \theta(t)$ in the $L^2(\Omega)$-norm.
\begin{theorem}\label{sup-conv-1}
For $0\le \delta\le 2$, we have
\[\|\nabla\theta(t)\|^2 \leq Ch^4 t^{-\alpha(3-\delta)}d_\delta^2(u_0,f),\quad {\rm  for}~~ t\in(0,T]\,.\]
\end{theorem}
 \begin{proof}
 We start by  applying the operator $\I^{1-\alpha}$ to both sides of the elementary identity  $\Ba \Theta_2(t)=\I^\alpha \Theta_2'(t)+\omega_\alpha(t)\Theta_2(0^+),$ to notice that
\begin{equation}\label{eq: theta2 identity}\begin{aligned}
\I^{1-\alpha} \Ba \Theta_2(t)&=\I^{1-\alpha} \I^\alpha \Theta_2'(t)+\I^{1-\alpha} \omega_\alpha(t)\Theta_2(0^+)\\
&=\Theta_2(t)-\Theta_2(0^+)+\Theta_2(0^+)=\Theta_2(t).\end{aligned}
\end{equation}
Now, applying again the operator $\I^{1-\alpha}$ to both sides  \eqref{k2-1}, and using the above equality  as well as the identity  $t\I^\alpha\theta = \I^\alpha\Theta_1+\alpha\I^{1+\alpha}\theta$ (by Lemma \ref{Ia} \mbox{(b)}) to get
\[(\I^{1-\alpha} \dot\Theta_2,\chi)+A(\Theta_2,\chi)=(\I^{1-\alpha}(t^2\rho'), \chi)+ (1-\alpha)\Big(2A(\I \Theta_1,\chi)+\alpha A(\I^2\theta,\chi)\Big).\]
Set $\chi=\dot \Theta_2$ follows by integrating the resulting equation from 0 to $t$  to obtain
\begin{align*}
\int_0^t[(\I^{1-\alpha} &\dot\Theta_2,\dot \Theta_2)+A(\Theta_2,\dot \Theta_2)]\,ds
\\ &\le \int_0^t(\I^{1-\alpha}(s^2\rho'),\dot \Theta_2)\,ds+(1-\alpha)\int_0^tA(2\I \Theta_1+\alpha\I^2\theta,\dot \Theta_2)\,ds.\end{align*}
However, by the continuity property of the operator $\I^{1-\alpha}$  in \cite[Lemma 3.1]{MustaphaSchoetzau2014},
\[\Big|\int_0^t(\I^{1-\alpha}(s^2\rho'),\dot\Theta_2)\,ds\Big|\le C\int_0^t(\I^{1-\alpha}(s^2\rho'),s^2\rho')\,ds
+\int_0^t(\I^{1-\alpha}\dot\Theta_2,\dot\Theta_2)\,ds,\]
and so,
\[\int_0^tA(\Theta_2,\dot \Theta_2)\,ds\le C\int_0^t(\I^{1-\alpha}(s^2\rho'),s^2\rho')\,ds+(1-\alpha)\int_0^tA(2\I \Theta_1+\alpha\I^2\theta,\dot \Theta_2)\,ds.\]
Using the identity  $2\I \Theta_1(t)=\Theta_2(t)-\I \dot \Theta_2(t)$ and the inequality $\int_0^tA(\I \dot \Theta_2,\dot \Theta_2)\,ds\ge 0$, after some simplifications, we conclude that
\[\alpha\int_0^tA(\Theta_2,\dot \Theta_2)\,ds\le C\int_0^t\|\I^{1-\alpha}s^2\rho'\|\,\|s^2\rho'\|\,ds+\alpha(1-\alpha)\int_0^tA(\I^2\theta,\dot \Theta_2)\,ds.\]
Since
\[\int_0^t A(\Theta_2,\dot\Theta_2)\,ds=\frac{1}{2}\|\sqrt{{\bf a}}\nabla\Theta_2(t)\|^2
-2\int_0^t s\|\sqrt{{\bf a}}\nabla\Theta_1(s)\|^2\,ds, \]
we easily find that
\begin{multline}\label{sup-9-n}
\alpha\|\sqrt{{\bf a}}\nabla\Theta_2(t)\|^2 \le  4\alpha\int_0^t s\|\sqrt{{\bf a}}\nabla\Theta_1(s)\|^2\,ds
\\+C\int_0^t\|\I^{1-\alpha}s^2\rho'\|\,\|s^2\rho'\|\,ds +2\alpha(1-\alpha)\int_0^tA(\I^2\theta,\dot \Theta_2)\,ds.
\end{multline}
By Lemma \ref{theta},
\begin{equation}\label{kn1-n}
\int_0^t s\|\sqrt{{\bf a}}\nabla\Theta_1\|^2\,ds\leq Ct\int_0^t \|\nabla\Theta_1\|^2\,ds \leq C h^4 t^{4-\alpha(3-\delta)}d^2_\delta(u_0,f)\,.
\end{equation}
To estimate the second term on the RHS of \eqref{sup-9-n}, we use the bound of $\rho'$ given in \eqref{eq: rho} (with $m=2$), the formula
\begin{equation}\label{eq: Inu}
\I^\nu (t^{\mu-1})= t^{\nu+\mu-1}\Gamma(\mu),\quad{\rm for}~~\nu,\,\mu >0,\end{equation}
and then integrate
\begin{multline}\label{kn4-n}
\int_0^t(\I^{1-\alpha}(s^2\rho'),s^2\rho')\,ds \le C  h^4\int_0^t s^{2-\alpha-\alpha(2-\delta)/2} s^{1-\alpha(2-\delta)/2}\,ds\, d^2_\delta(u_0,f)\\
\le Ch^4 t^{4-\alpha-\alpha(2-\delta)}d^2_\delta(u_0,f)\,.
\end{multline}
For the last term on the RHS of \eqref{sup-9-n}, we apply $\I^{2-\alpha}$ to (\ref{sup-3 ph}) to obtain $A(\I^2 \theta,\chi)= (\I^{2-\alpha}e,\chi).$ Hence, integrating by parts, we find that
\begin{align*}
\int_0^tA(\I^2\theta,\dot \Theta_2)\,ds &=\int_0^t(s^2\I^{2-\alpha}e,\theta')\,ds\\
&=(\I^{2-\alpha}e(t),\Theta_2(t))-\int_0^t(2\I^{2-\alpha}e+s\I^{1-\alpha}e,\Theta_1)\,ds\,.\end{align*}
Then,  by using the estimate of $\theta$ in Lemma \ref{H1},  \eqref{eq: bound of B1}, and the estimate of $e$ in Theorem \ref{thm: smooth and nonsmooth},  we conclude  after integrating and using the formula in \eqref{eq: Inu}, that
\[\begin{aligned}
\Big|\int_0^tA(\I^2\theta,\dot \Theta_2)\,ds\Big|&\le t^2\|\I^{2-\alpha}e(t)\|\,\|\theta(t)\|+2t\int_0^t\|\I^{2-\alpha}e+s\I^{1-\alpha}e\|\,\|\theta\|\,ds\\
&\le Ch^4 t^{4-\alpha-\alpha(2-\delta)}d_\delta^2(u_0,f)\,.
\end{aligned}\]
A substitution of the estimates \eqref{kn1-n}, \eqref{kn4-n} and the above one   in (\ref{sup-9-n}), follows by using \eqref{eq: A positive} and the identity $\theta(t)=t^{-2}\Theta_2$ yield the desired estimate.  $\quad \Box$ \end{proof}

Noting that, by using the estimates of $\rho$, $\rho'$ and $e$ from  Remark \ref{rem: 0} in the inequality \eqref{eq: bound theta}, we observe
\[\int_0^{t} \|\nabla\Theta_1\|^2ds \leq C h^2 t^{3-\alpha(2-\delta)}D^2_{\delta,\alpha}(u_0,f,h/t)\,.\]
Hence, by following the steps in Theorem \ref{sup-conv-1}, and using the above  bound instead of  Lemma \ref{theta}, and the  bounds of $\rho'$ and $e$ achieved in Remark \ref{rem: 0},  we deduce that
\[  \|\nabla \theta(t)\|^2 \leq
 C\, h^2\,t^{-\alpha(2-\delta)}D^2_{\delta,\alpha}(u_0,f,h/t)\,.\]
Therefore, from the inequality $\|\nabla (u_h-u)(t)\|\le \|\nabla \theta(t)\| +\|\nabla \rho(t)\|$, the above bound, the  bound of  $\eta$ in \eqref{rho-estimate} (with $j=1$ and $m=2$)  and the regularity property (\ref{eq: regularity property}), we have the following result.
 \begin{theorem} \label{thm: H1 bound}
 Let $u$ and $u_h$  be the solutions of $(\ref{a})$ and $(\ref{semi})$, respectively, with $u_h(0)=P_h u_0$. For $u_0 \in \dot H^\delta(\Omega)$, for $t\in (0,T]$,  we have
\[ \|\nabla(u-u_h)(t)\| \leq C\, h\,t^{-\alpha(2-\delta)/2}\tilde d_\delta(u_0,f)\times \begin{cases} 1,~~~&~~0\le \delta \le 1,\\
 \max\{1,(h/t)^{\alpha(\delta-1)/2}\},~~~&~~1< \delta \le 2\,.\end{cases}\]
\end{theorem}
\begin{remark}\label{rem: 3}
The estimate in Theorem \ref{sup-conv-1} suggests that one can achieve a higher convergence rate for $\nabla(u_h-u)$ if an improved estimate of the error
$\nabla(R_hu-u)$ can be derived.  This could be achieved using a superconvergent recovery procedure of the gradient, which is possible on special meshes and  for solutions in $H^3(\Omega)$ for each $t\in (0,T]$. Examples of special meshes exhibiting superconvergence property are provided in  \cite{MN-87}. Therein, the authors  introduced an operator $G_h$ which postprocesses  $\nabla R_h u(t)$  with the following properties:
\begin{itemize}
\item[(i)] If $u(t)\in H^3(\Omega)$, then $ \|\nabla u(t)-G_h(R_h u)(t)\|\leq Ch^2\|u(t)\|_{H^3(\Omega)}.$
 \item[(ii)] For $\chi\in V_h$, we have $\|G_h(\chi)\|\leq C\|\nabla \chi\|.$
 \end{itemize}
Now,  if ${\cal T}_h$ is a triangulation of $\Omega$ such that these results are satisfied, then using
$$\|\nabla (u-u_h)(t)\|\le \|(\nabla u-G_h(R_h u))(t)\| +\|G_h(R_h u-u_h)(t)\|+ \|\nabla \theta(t)\|,$$  (i) and (ii), Theorem \ref{sup-conv-1}, and the inequality  $\|u(t)\|_{H^3(\Omega)} \le Ct^{-\alpha(3-\delta)/2}d_\delta(u_0,f)$ for $1/2 < \delta\leq 2$, it is clear that   the bound below holds for $t\in (0,T]$,
\[ \|\nabla(u-u_h)(t)\| \leq C h^{2} t^{-\alpha(3-\delta)/2}d_\delta(u_0,f), \qquad 1/2< \delta\leq 2\,. \quad \Box\]
\end{remark}

For $t\in (0,T]$, we show in the next theorem that the superconvergence result of $\nabla \theta$ in Theorem \ref{sup-conv-1} can be used to establish a  quasi-optimal (due to the presence of the logarithmic factor) convergence rate in the stronger $L^\infty(\Omega)$-norm.  Recall that, in the limiting case $\alpha=1$, the fractional diffusion problem \eqref{a} reduces to the classical diffusion equation. For $\delta=0,$ it is known that the logarithmic factor $\ell_h$  in this case is of order $2$, see \cite[Theorem 6.10]{thomee2006}, while it is of order $5/2$ in the theorem below. So, one can argue that the order of $\ell_h$ is not sharp.
\begin{theorem}\label{thm: smooth and nonsmooth-5}
  Let $u$ and $u_h$  be the solutions of $(\ref{a})$ (with $f\equiv 0$)  and $(\ref{semi})$, respectively, with $u_h(0)=P_h u_0$. Under the quasi-uniformity condition on $V_h$, for $t \in (0,T]$, we have
\[\|(u-u_h)(t)\|_{L^{\infty}(\Omega)} \leq  C h^2\ell_h^{5/2} t^{-\alpha(3-\delta)/2}\Big(\|u_0\|_\delta+\|u_0\|_{L^\infty(\Omega)}\Big)\quad {\rm for}~~0\le \delta\le 2.\]
\end{theorem}
\begin{proof} By the  Ritz projection error result  \cite[Equation (6.81)]{thomee2006}  and  the Agmon-Douglis-Nirenberg \cite{AgmonDouglisNirenberg1959} regularity estimate $\|\phi\|_{W^{2,p}(\Omega)}\le Cp\|\mathcal L \phi\|_{L^p(\Omega)}$ for $\phi \in W^{2,p}(\Omega) \cap H^1_0(\Omega)$ with $2\le p<\infty$, we have
\begin{equation}\label{eq: infty 1}
\|\rho(t)\|_{L^\infty(\Omega)}\leq C\ell_h h^{2-2/p}\|u(t)\|_{W^{2,p}(\Omega)} \leq C h^{2-2/p}\ell_h^{5/2}\, p \|\mathcal L u(t)\|_{L^p(\Omega)}\,.\end{equation}
{A time integration of  both sides of \eqref{a} ($f\equiv 0$), gives $\I u'(t)+{\mathcal L} (\I^\alpha u(t)-\I^\alpha u(0))=0.$ Since $\|u(t)\|\le C\|u_0\|$ (by \eqref{eq: regularity property}), $\|\I^\alpha u(t)\|\le C t^\alpha$ for $t>0$. Then  $\I^\alpha u(0)=0$ and so, $\I u'(t)+ \I^\alpha {\mathcal L} u(t)=0.$ Applying the operator $\I^{1-\alpha}$ to both sides,
\begin{equation}\label{eq: new identity}\I^{2-\alpha}u'(t)+\I {\mathcal L} u(t)=0.\end{equation}
 By  Lemma \ref{Ia} (a) (with ${\mathcal L}u$ in place of $\theta$) and \eqref{eq: new identity}, we have
 \begin{align*}  \I^{1-\alpha}(t\Ba  {\mathcal L} u(t))&=\I^{1-\alpha}(\Ba{\mathcal L} (tu(t)))-(1-\alpha)\I {\mathcal L} u(t)
 \\&=\I^{1-\alpha}\Big(\Ba{\mathcal L} (tu(t))+(1-\alpha)\I u(t)\Big). \end{align*}
 Using the identities $\I^{1-\alpha}(t\Ba  {\mathcal L} u(t))=-\I^{1-\alpha} (tu'(t))$ (follows from problem \eqref{a} with $f\equiv 0$)) and $\I^{1-\alpha}(\Ba{\mathcal L} (tu(t)))=\mathcal L (tu(t))$ (follows from \eqref{eq: theta2 identity} with ${\mathcal L} (tu(t))$ in place of $\Theta_2$), we find that
 \begin{align*}
   t{\mathcal L} u(t)&=-\I^{1-\alpha} \Big((tu'(t))+(1-\alpha) \I u'(t)\Big).\end{align*}}
Hence, by the embedding inequality ($\|v\|_{L^p(\Omega)}\leq C \sqrt{p} \|\nabla v\|$ for $v\in H^1_0(\Omega)$), the regularity property in \eqref{eq: regularity property-2}, and the property  $\|u(t)\|_{L^\infty(\Omega)}\le C\|u_0\|_{L^\infty(\Omega)}$, we have
 \begin{align*}
\|\mathcal L u(t)\|_{L^p(\Omega)}   &\le  C\sqrt{p}\I^{1-\alpha}(\|tu'(t)\|_1)+C\begin{cases}\I^{1-\alpha}(\|u(t)-u_0\|_{L^\infty(\Omega)}),~&0\le \delta\le 1\\    \sqrt{p}\I^{2-\alpha}(\|u'(t)\|_1),~&1<\delta\le 2\end{cases}\\
   &\le C\sqrt{p}\,t^{1-\alpha(3-\delta)/2}\Big(\|u_0\|_\delta+\|u_0\|_{L^\infty(\Omega)}\Big),\quad 0\le \delta\le 2\,.\end{align*}
for any $2\le p<\infty.$  Inserting the above bound in  \eqref{eq: infty 1} implies that
\[\|\rho(t)\|_{L^\infty(\Omega)}\le C h^{2-2/p}\ell_h^{5/2}\, p^{3/2} t^{-\alpha(3-\delta)/2}\Big(\|u_0\|_\delta+\|u_0\|_{L^\infty(\Omega)}\Big),\quad0\le  \delta\le 2.\]
On the other hand, by the discrete Sobolev inequality and the estimate in Theorem \ref{sup-conv-1},  we observe that
\[\|\theta(t)\|_{L^\infty(\Omega)}\leq C\ell_h^{1/2}\|\nabla \theta(t)\|\leq  C h^2\ell_h^{5/2} t^{-\alpha(3-\delta)/2}\|u_0\|_\delta\quad {\rm for}~~0\le \delta\le 2.\]
Finally, choose $p=|\ln h|$, and the desired convergence result follows then  from $\|(u_h-u)(t)\|_{L^\infty(\Omega)}\leq \|\theta(t)\|_{L^\infty(\Omega)}
+\|\rho(t)\|_{L^\infty(\Omega)}$, and the above two bounds.$\quad \Box$
\end{proof}
\begin{remark}\label{rem: non-zero f}
One can extend the achieved results in Theorem \ref{thm: smooth and nonsmooth-5} to the case of non-zero source term $f$, assuming some regualrity assumptions such as $\I^{1-\alpha} (t\|f(t)\|_1)\le C$ and $\I(\|f(t)\|_{L^\infty(\Omega)})\le C.$
\end{remark}

\section{Numerical results}\label{sec: Numerical}
In this section, we focus on testing the achieved theoretical convergence results in Theorem \ref{thm: smooth and nonsmooth-5}. For the numerical illustration of the error bounds in Theorems  \ref{thm: smooth and nonsmooth} and \ref{thm: H1 bound}, one can follow the convention in \cite[Section 6]{JLZ2013}. To this end, we choose ${\cal{L}} = -\nabla^2$, $f\equiv 0$, $\alpha=0.75$,  $T=0.5$,  and $\Omega=(0,1)\times (0,1)$ in problem  \eqref{a}. The orthonormal eigenfunctions and corresponding eigenvalues of~${\cal{L}}$ are
\[\phi_{mn}(x,y)=2\sin(m \pi x)\sin(n \pi y) \quad\text{and}\quad \lambda_{mn}=(m^2+n^2)\pi^2\quad{\rm for}~~ m\,,n=1, 2, \ldots.\]
Separation of variables yields the series representation solution of problem \eqref{a}:
\begin{equation}\label{eq: u series}
u(x,y,t)=\sum_{m,n=1}^\infty (u_0, \phi_{mn}) E_{\alpha}(-\lambda_{mn} t^{\alpha})\phi_{mn}(x,y),
\end{equation}
where $E_{\alpha}(t):=\sum_{p=0}^\infty\frac{t^p}{\Gamma(\alpha p+1)}$ is the Mittag-Leffler function.

To compute the semidiscrete solution $u_h$, we discretize in time by the mean of generalized Crank-Nicolson   scheme \cite{Mustapha2011}, this will then define the following scheme:
 \begin{equation*}
\tau_n^{-1}(u_h^n-u_h^{n-1},v_h)+ A(\I^\alpha \bar u_h(t_n)-\I^\alpha \bar u_h(t_{n-1}),v_h)= 0\quad \forall v_h\in V_h,
\end{equation*}
for $1\le n\le N$,  where $N$ is the number of time mesh subintervals ($0=t_0<t_1<\ldots<t_N=T$), $\tau_n$ is the $n$th time step size. Here
$u_h^n \approx u_h(t_n)$ and $\bar u_h(s)=\frac{1}{2}(u_h^j+u_h^{j-1})$ when $s \in (t_{j-1},t_j)$ for $j\ge 2,$ while $\bar u_h(s)=u_h^1$ on the subinterval $(0,t_1).$ The modification on the first subinterval ensures that $\bar u_h$  does not depend on $u_h^0$  which is necessary for our numerical scheme in cases when $u_0$ is not sufficiently regular.

Following the convergence analysis in \cite{Mustapha2011}, we concentrate the time step near $t=0$ to compensate for the singular behaviour of the solution $u$ of problem \eqref{a}. So, we let $t_n=(n/N)^\gamma T$ for some fixed $\gamma\ge 1$ that will be chosen appropriately.  For the spatial partition of $\Omega$, let $\cT_h$ be a family of uniform  triangular meshes   with diameter ~$h=\sqrt{2}/M$ obtained from uniform $M$-by-$M$ square meshes by cutting each mesh square into two triangles. For measuring the error  at each time node $t_n$, we let ${\mathcal N}_h$ be the set of all triangular nodes of the mesh family  $\cT_{h_s}$ where the diameter $h_s$ is half the diameter of the finest mesh $\cT_h$ in  our spatial iterations, for instance, $h_s=\sqrt{2}/128$ in Tables \ref{table 1}--\ref{table 3} as well as in Figures \ref{fig1}--\ref{fig3}. To measure the errors, define the discrete-space maximum norm: $|\|v\||:=\max\{|v({\bf x})|,~{\bf x}\in {\mathcal N}_h\}\,.$ Thus,  for large values of $M$,  $|\|u_h^n-u(t_n)\||$ approximates the error $\|u_h^n-u(t_n)\|_{L^{\infty}(\Omega)}$.

\begin{table}
\begin{center}
\begin{tabular}{|c|cc|}
\hline
$M$& $Error$& $CR$\\
\hline
 4& 1.2759e-02&       \\
 8& 3.3749e-03& 1.9186\\
16& 8.7940e-04& 1.9402\\
32& 2.2284e-04& 1.9805\\
64& 5.6414e-05& 1.9819\\
\hline
\end{tabular}
\caption{Behavior of the uniform error $\max_{n=1}^N  |\|u_h^n-u(t_n)\||$ and the associated convergence rates as the number of spatial mesh elements increases. In each case, we use $1000$  time subintervals.}
\label{table 1}
\end{center}
\end{table}
\begin{figure}
\begin{center}
\includegraphics[width=10cm, height=6cm]{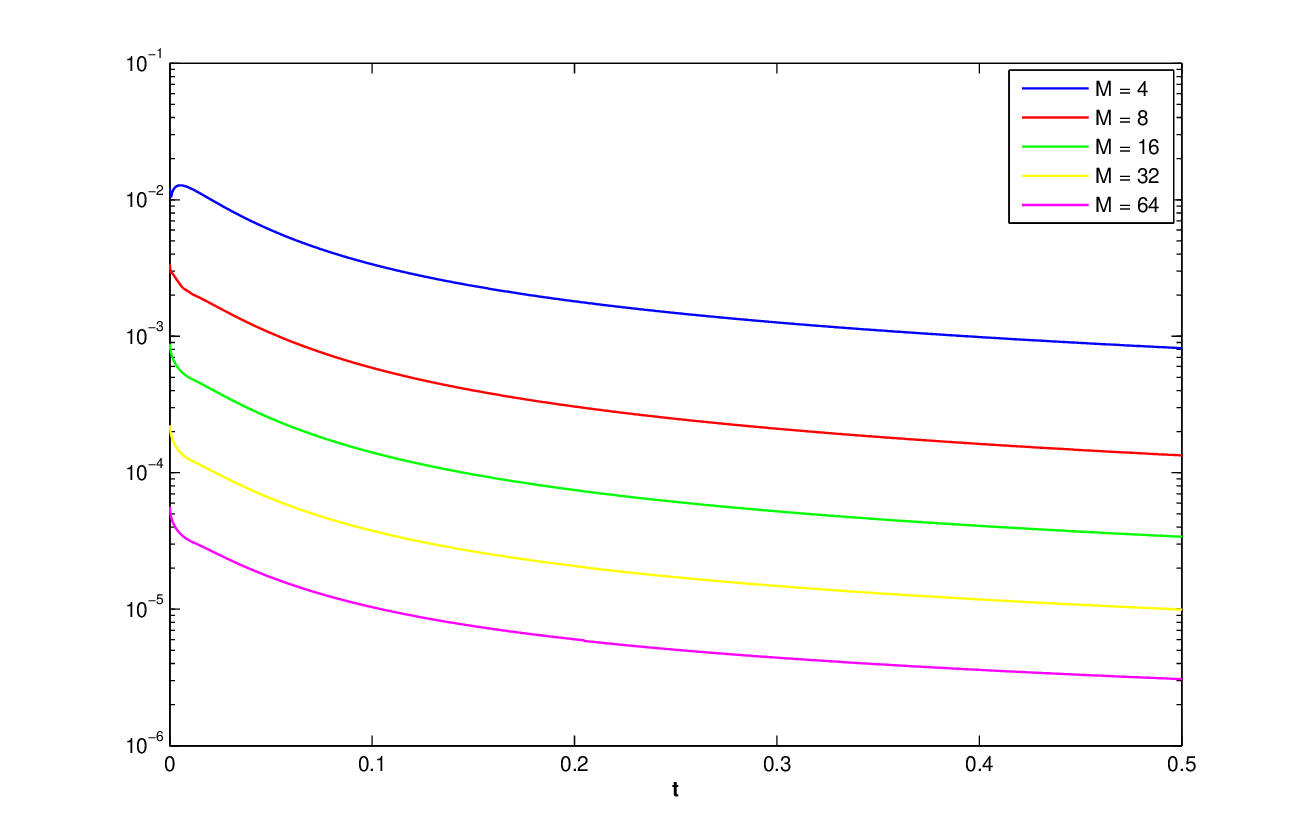}
\caption{The error $|\|u_h^n-u(t_n)\||$ as a function of~$t_n$ for Example 1.}
\label{fig1}
\end{center}
\end{figure}

In  {\bf Examples 1-3}, we choose $\gamma=1.6$ and  refine the time steps so that the spatial errors are dominant. We evaluate the exact solution $u$ of problem \eqref{a} by truncating the Fourier series in \eqref{eq: u series} after $60$ terms.

{\bf Example 1.}  Choose  $u_0(x,y)=xy(1-x)(1-y)$. The Fourier sine coefficients are:
\[(u_0,\phi_{mn})=8(1-(-1)^m)(1-(-1)^n)(mn\pi^2)^{-3},\quad {\rm for}~~m,n=1,\,2,\ldots.\]
The initial data  $u_0 \in \dot H^2(\Omega)\cap L^\infty(\Omega).$ Thus, by Theorem \ref{thm: smooth and nonsmooth-5} ($\delta=2$), for each time step $t_n$,  we expect convergence rate of order $h^2 \ell_h^{5/2}t_n^{-\alpha/2}$ in the  $L^\infty(\Omega)$-norm. Figure \ref{fig1}  shows how the error varies with $t$ for a sequence of solutions obtained by successively doubling the spatial mesh elements, using a log scale. (The same time mesh with $N=1000$ subintervals  was used in all cases). In Table \ref{table 1}, we listed the time-space maximum error and its associated convergence rate ($CR$), where second order optimal convergence rates was observed (ignoring the logarithmic factors). So, the  influence of the  coefficient $t_n^{-\alpha/2}=t_n^{-3/8}$ is absent. This is probably due to the fact the $u_0$   belongs to the smoother space ${\mathcal C}^2(\overline \Omega)\cap {\mathcal C}_0(\overline \Omega)$, where an $O(h^2\ell_h^2)$ rate of convergence is expected, \cite[Theorem 4.2]{MT2010b}.
\begin{table}
\begin{center}
\begin{tabular}{|c|cc|cc|cc|cc|}
\hline
$M$&\multicolumn{2}{c|}{$\mu=0$ }
&\multicolumn{2}{c|}{$\mu=0.25$}
&\multicolumn{2}{c|}{$\mu=0.5$}&\multicolumn{2}{c|}{$\mu=75$}\\
\hline
 4& 3.008e-02& & 9.521e-03&& 3.610e-03&& 1.597e-03&\\
 8& 1.054e-02&  1.513  & 1.412e-03& 2.754& 5.342e-04& 2.757& 2.401e-04& 2.734\\
16& 5.441e-03&  0.954  & 4.112e-04& 1.779& 1.279e-04& 2.062& 5.678e-05& 2.080\\
32& 1.876e-03&  1.536  & 1.391e-04& 1.564& 3.344e-05& 1.936& 1.513e-05& 1.908\\
64& 8.667e-04&  1.114  & 6.425e-05& 1.114& 8.598e-06& 1.959& 4.055e-06& 1.900\\
\hline
\end{tabular}
\caption{The weighted error $E_\mu$ and the convergence rates, as the number of spatial mesh elements increases, for different choices of the
power  exponent $\mu$. In each case, we use $1300$  time subintervals. }
\label{table 2}
\end{center}
\end{table}
\begin{figure}
\begin{center}
\includegraphics[width=10cm, height=6cm]{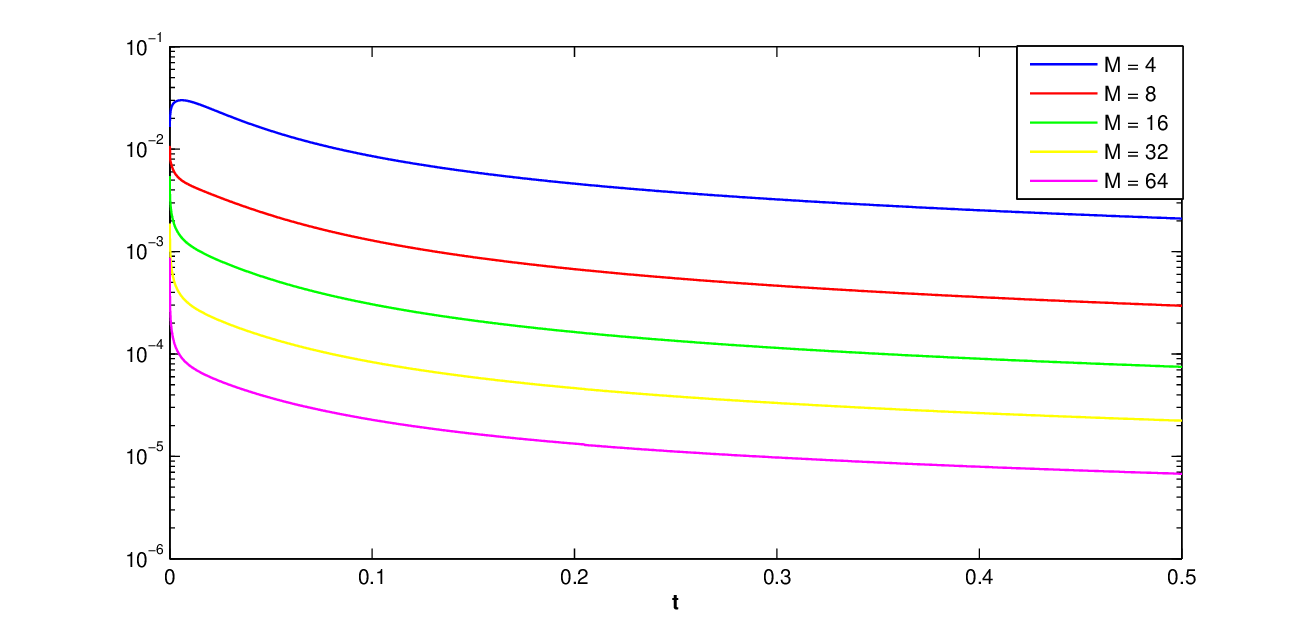}
\caption{The error $|\|u_h^n-u(t_n)\||$ as a function of~$t_n$ for Example 2.}
\label{fig2}
\end{center}
\end{figure}

\begin{table}
\begin{center}
\begin{tabular}{|c|cc|cc|cc|cc|}
\hline
$M$&\multicolumn{2}{c|}{$\mu=0$ }
&\multicolumn{2}{c|}{$\mu=0.5$}
&\multicolumn{2}{c|}{$\mu=0.75$}&\multicolumn{2}{c|}{$\mu=1$}\\
\hline
 8& 9.7501e-01&& 8.545e-03&& 2.245e-03&& 1.525e-03&\\
16& 7.0054e-01&  0.4769& 3.852e-03& 1.150& 6.809e-04& 1.721& 4.783e-04& 1.672\\
32& 3.2311e-01&  1.1164& 1.776e-03& 1.117& 2.000e-04& 1.767& 1.442e-04& 1.730\\
64& 1.5301e-01&  1.0783& 8.409e-04& 1.078& 6.234e-05& 1.682& 4.945e-05& 1.544\\
\hline
\end{tabular}
\caption{The weighted error $E_\mu$ and the convergence rates, as the number of spatial mesh elements increases, for different choices of the
power weight exponent $\mu$. In each case, we use $1300$  time subintervals. }
\label{table 3}
\end{center}
\end{table}

{\bf Example 2.}  Choose  $u_0(x,y)=g(x)g(y)$ where $g(z)=z$ on $[0,1/2)$ while $g(z)=1-z$ on  $(1/2,1]$, which is less smooth, then the considered $u_0$ in the previous example. One can verify that $u_0$ has the Fourier sine coefficients $(u_0,\phi_{mn})=2(1-(-1)^m)(1-(-1)^n)(mn\pi^2)^{-2}(-1)^{mn}$ for $m,n=1,\,2,\ldots.$ The function $u_0 \in \dot H^{1+\epsilon}(\Omega)$ for $0\le \epsilon<1/2$. So, by Theorem \ref{thm: smooth and nonsmooth-5} ($\delta<1.5$), for each $t_n$, we expect $O(h^2\ell_h^{5/2})t_n^{-3\alpha/4}$ convergence rates in the  $L^\infty(\Omega)$-norm. As in Figure \ref{fig1},  Figure \ref{fig2} shows how the error varies with $t$  for a sequence of solutions obtained by doubling the spatial mesh elements. (The time mesh with $N=1300$ subintervals  was used in all cases). Table \ref{table 2} provides an alternative view of this data, listing the time-space maximum weighted error $E_\mu:=\max_{n=1}^N t_n^\mu |\|u_h^n-u(t_n)\||$  and its associated convergence rate $CR$. As expected, ignoring the logarithmic factors, the convergence rate is $2$ when $\mu \ge 3\alpha/4 \approx 0.56$, but  the rate deteriorates for smaller values of $\mu$ (relatively far from $3\alpha/4$).

{\bf Example 3.}  Choose $u_0(x,y)=1$, and so $u_0$ has the Fourier sine coefficients $(u_0,\phi_{mn})=2(1-(-1)^m)(1-(-1)^n)(mn\pi^2)^{-1}$ for $m,n=1,\,2,\ldots.$ The initial data function  $u_0 \in \dot H^{\epsilon}(\Omega)\cap L^\infty(\Omega)$ for $0\le \epsilon<1/2$. As in the previous example, Figure \ref{fig3} shows a consistent decaying in the errors by doubling the number of spatial mesh elements. Another observation is the large impact of the very limited regularity of $u_0$ on the errors near $t=0$ in this example. For better justifications of this, see Table \ref{table 3} where the difference between the maximum error $E_0$ and the weighted error $E_1$ is very substantial, we also observed very good improvements in the convergence rates $CR$, but not yet optimal due to the time discretization.

\begin{figure}
\begin{center}
\includegraphics[width=10cm, height=6cm]{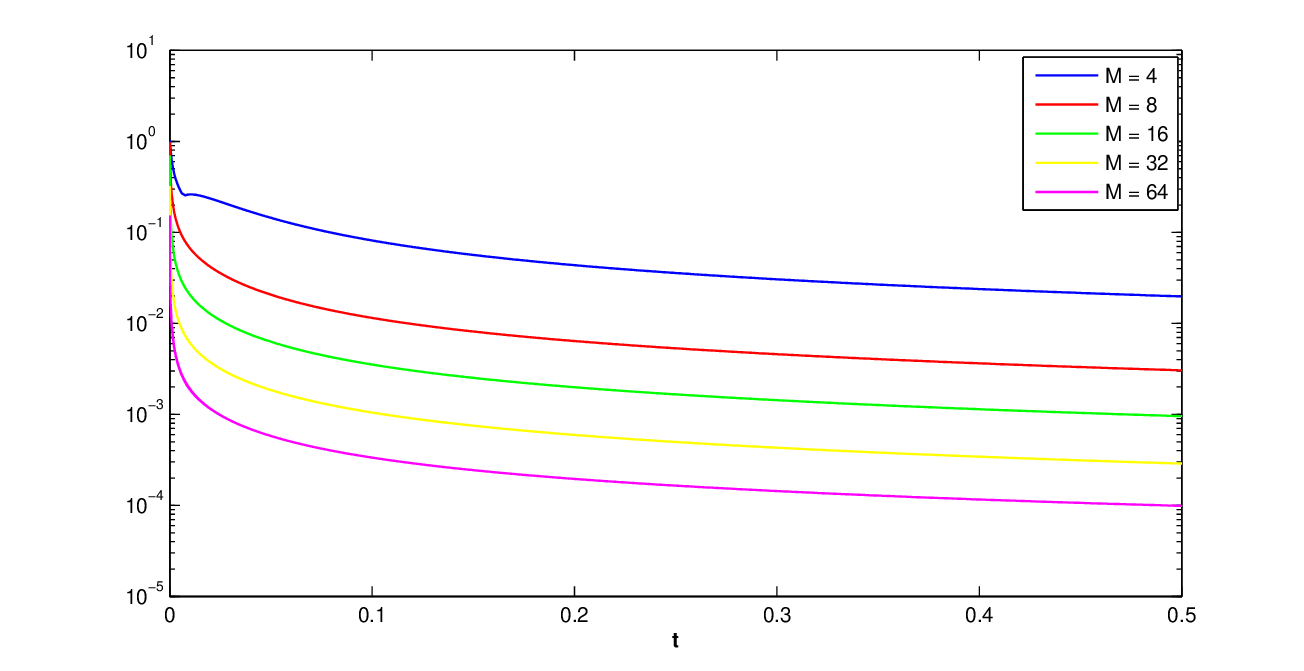}
\caption{The error $|\|u_h^n-u(t_n)\||$ as a function of~$t_n$ for Example 3.}
\label{fig3}
\end{center}
\end{figure}

\end{document}